\documentclass[10pt,oneside]{amsart}
\usepackage{amsmath,amsthm,amssymb,epic,eepic,bm}
\usepackage{graphicx,stmaryrd}
\usepackage{amsmath,amsthm,amssymb,epic,eepic,bm,mathabx,mathtools}
\usepackage{eqlist,eqparbox,tikz,tikz-cd}

\begin{document}

\newtheorem{defin}{Definition}
\newtheorem{theor}[defin]{Theorem}
\newtheorem{prop}[defin]{Proposition}
\newtheorem{lemma}[defin]{Lemma}
\newtheorem{exam}[defin]{Example}
\newtheorem*{claim}{Claim}
\title[A new representation of finite Hoops]{\bf A new representation of finite Hoops using a new type of product of structures}
\author {Michal Botur}

\thanks{Author acknowledge the support by GA\v CR 24-14386L and by IGA, project P\v rF 2024 011 and 2025 008.} 
\address{Palack\' y University Olomouc, Faculty of Sciences,  17. listopadu 1192/12, Olomouc 771 46, Czech Republic}
\email{michal.botur@upol.cz}
\keywords{wreath product, residuated lattice}
\subjclass[2010]{Primary 06D35, Secondary 03B50}

\begin{abstract}
We define a new type of product of hoops which,
in the case of finite hoops, can decompose an arbitrary hoop $\mathbf A$ into
a filter $F$ and the corresponding homomorphic image $\mathbf A/F$.
Since our new product is associative up to
isomorphism, we can show that every finite hoop is representable
as a product of finite MV-chains.  
\end{abstract}

\maketitle

\section{Motivation}

Hoops were introduced in an unpublished manuscript by
Büchi and Owens, grounded in the work of Bosbach on partially ordered monoids
\cite{Bos}. Hoops have been thoroughly studied and we have a very good
understanding of their structure (see for example \cite{BlFef,Haj,Bo}).
This paper provides a new representation of finite hoops, whose general
idea grew out of investigations of various possibilities of wreath and
semidirect products of residuated structures.
These ideas, in their initial form, occurred to me during my
internship with Constantine Tsinakis, and it was this eminent scholar who gave the
initial impetus to this research. I was fortunate to be able to work with
Constantine, and receive support from him. This article is a token of respect and
gratitude.

Another motivation and inspiration for the notions introduced in the main part
was the deep theory of semidirect and wreath products of semigroups,
which culminates in the famous Krohn-Rhodes theorem \cite{KR}. 

A semigroup action $\mathcal A=(X,\mathbf A,\bullet)$ is a structure where
$\mathbf A=(A;\cdot,1)$ is a semigroup, $X$ is a set and 
$\bullet\colon X\times A\longrightarrow X$ is a map, called the \emph{right action},
which satisfies $x\bullet (a\cdot b)=(x\bullet a)\bullet b$, for all $x\in X$
and $a,b\in A$.
For any semigroup action $\mathcal B= (X,\mathbf B,\bullet)$ and
any semigroup $\mathbf A=(A;\cdot)$ we can define the following semigroup
$$
\mathbf A\wr\mathcal B = (A^X\times B,\cdot)
$$
where
$$
(\overline a_1,b_1)\cdot (\overline a_2,b_2)=(\overline a_1\cdot (\overline
a_2\asterisk b_1),b_1\cdot b_2),
$$
$$
(\overline a_2\asterisk b_1)(x) = \overline a_2(x\bullet b_1)
$$
for all  $(\overline a_1,b_1),(\overline a_2,b_2)\in A^X\times B$ and any $x\in X$. 
The semigroup $\mathbf A\wr\mathcal B$
is called a \emph{wreath product} of $\mathbf{A}$ and $\mathcal{B}$.

There is a certain asymmetry in this notion of wreath product, since $\mathbf{A}$
does not need to act on anything. But wreath product can be defined for two
semigroup actions as well. If $\mathcal A=(X,\mathbf A,\bullet)$
and $\mathcal B=(Y,\mathbf B,\bullet)$ are semigroup actions then we
define  $$\mathcal A\wr \mathcal B = (X\times Y,\mathbf 
 A\wr\mathcal B,\bullet)$$
 such that $(x,y)\bullet (\overline a,b)= (x\bullet \overline a(y),y\bullet b)$
 for any $(x,y)\in X\times Y$ and $(\overline a,b)\in A^Y\times B$.
Defined this way, wreath product can be seen as an operation in the class of all
semigroup actions.

The first (and at first sight perhaps surprising) result is associativity of
such wreath product. More precisely, if $\mathcal A=(X,\mathbf A,\bullet)$,
$\mathcal B=(Y,\mathbf B,\bullet)$ and $\mathcal C=(Z,\mathbf C,\bullet)$ are
semigroup actions then  
$$
\mathcal A\wr (\mathcal B\wr\mathcal C)\cong(\mathcal A\wr\mathcal B)\wr
\mathcal C
$$
where the isomorphism ($\cong$) naturally means the existence of a pair of
bijections between sets together with a monoid isomorphism, preserving actions.
Associativity is an important property guiding the approach taken in this work,
so let me analyse it in a little more detail. First of all, for sets on
which the semigroups act, the associativity is the usual associativity of direct
product of sets.   

For semigroups, their wreath product corresponds to the \emph{cascade
product} of automata. The second ``semigroup factor''
forms a kind of skeleton of the wreath product, and each element $x$ of
the skeleton induces an automorphism on the first semigroup. Denoting this
automorphism, in a category-theory fashion, by $-\asterisk x$, we see that
in the wreath product any such $x$ is replaced by
the ``$-\asterisk x$ transformed'' copy of the semigroup.
This is illustrated in Figure~\ref{F1}.
Associativity, as illustrated in Figure~\ref{F2}, states that 
such an cascade expansion is independent on the order in which it is
performed.

The Krohn-Rhodes theorem (and all of the subsequent theory) is central to our
motivation for introducing new products for hoops. Namely, Krohn-Rhodes theorem
shows that all monoid actions can in some sense be decomposed into simple groups
and three element flip-flop monoids (for more details we refer the reader to
\cite{KR}). Similarly we show that any finite hoop is a ``new type'' product of
MV-chains. 

We want to emphasize that the introduction of this new product for hoops is
not meant to be a mere syntactic analogy with the semigroup case.
Rather, we look for a ``product-like'' operation, which would be
at the very least associative in a natural way, and we wish to proceed
towards a representation of hoops by way of some result that would say:
\emph{every finite hoop is constructible as a new product of certain
new-product-irreducible ones}.

\vskip 10pt
\begin{figure}
    \centering
    \includegraphics[width=0.6\linewidth]{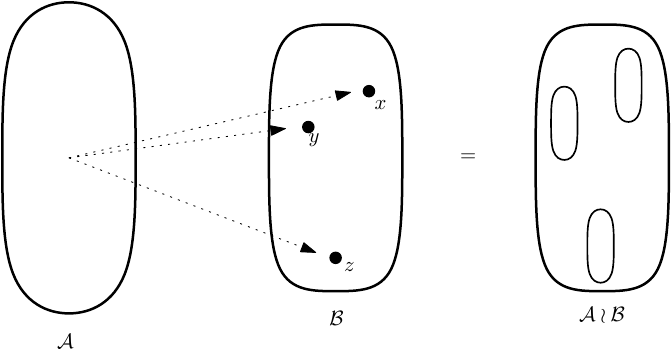}
    \caption{Idea of a wreath product of semigroup actions.}
    \label{F1}
\end{figure}
\vskip 10pt

\vskip 10pt
\begin{figure}
    \centering
    \includegraphics[width=0.9\linewidth]{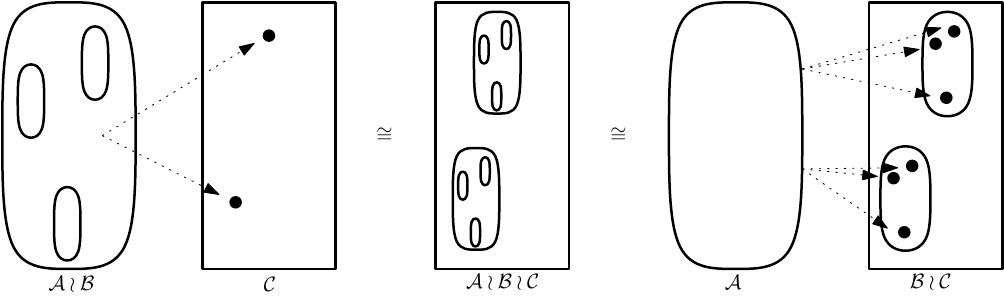}
    \caption{Associativity of a wreath product.}
    \label{F2}
\end{figure}
\vskip 10pt

For the sake of completeness, we mention that the possibility of a semidirect product similar to the one considered in this paper is also dealt with in \cite{CL,la}. However, neither of these articles focuses on associativity, which is our main motivation.

\section{Introduction}
A commutative ordered monoid $(A,\cdot,1,\leq)$ is ordered \emph{naturally} if
the order relation satisfies : $a\leq b$ if and only if exists $c\in A$ such
that $b\cdot c=a$ (for all $a,b\in A$). A \emph{hoop} can be defined as
naturally ordered commutative monoid possessing the binary operation
$\rightarrow$ such that the adjointnes property 
$$
a\cdot b\leq c \text{ if and only if } a\leq b\rightarrow c
$$ 
holds for all $a,b,c$. The usual equivalent definition that follows is more
algebraic but less clear. 
\begin{defin}
A \emph{hoop} is the algebra $\mathbf A=(A;\cdot,\rightarrow,1)$ of the type $\langle 2,2,0\rangle$, where $(A;\cdot,1)$ is a commutative monoid and satisfying the identities
\begin{itemize}
    \item[(H1)] $x\rightarrow x=1$,
    \item[(H2)] $(x\cdot y)\rightarrow z = x\rightarrow (y\rightarrow z),$
    \item[(H3)] $x\cdot (x\rightarrow y)=y\cdot (y\rightarrow x)$. 
\end{itemize}
\end{defin}

If $\mathbf A=(A;\cdot,\rightarrow,1)$ is a hoop then we introduce the relation
$\leq$ on $A$ by $x\leq y$ if $1=x\rightarrow y$ (for any $x,y\in A$). It is
well known that $(A;\leq)$ is an ordered set with the top element $1$ such that
there exists an infimum $x\wedge y=x\cdot (x\rightarrow y)$ for any $x,y\in A$.
This property is known as \emph{divisibility}; in presence of the other defining
identities it is equivalent to (H3).

\begin{lemma}[\cite{Haj}]
    If $\mathbf A=(A;\cdot,\rightarrow,1)$ is a hoop and $x,y,z\in A$ then it satisfies
    \begin{itemize}
        \item[(i)] if $x\leq y$ then $x\cdot z\leq y\cdot z,$ $z\rightarrow x\leq z\rightarrow y$ and $y\rightarrow z\leq x\rightarrow y,$
        \item [(ii)] $x\cdot y\leq z$ if and only if $x\leq y\rightarrow z,$
        \item [(iii)] $x\rightarrow (y\wedge z)=(x\rightarrow y)\wedge (x\rightarrow z),$
        \item[(iv)] if the supremum $y\vee z$ exists then also $(x\cdot y)\vee (x\cdot z)$ exists and $x\cdot (y\vee z)= (x\cdot y)\vee (x\cdot z),$
        \item[(v)] if the supremum $x\vee y$ exists then   $(x\vee y)\rightarrow z= (x\rightarrow z)\wedge (y\rightarrow z).$
     \end{itemize}
\end{lemma}

\begin{lemma}
    If $\mathbf A=(A;\cdot,\rightarrow,1)$ is a hoop then $x\cdot (y\wedge z)= (x\cdot y)\wedge (x\cdot z)$ holds for any $x,y,z\in A.$
\end{lemma}
\begin{proof}
    It is clear that $x\cdot (x\rightarrow (x\cdot z))\leq x\cdot z$ and $z\leq x\rightarrow (x\cdot z)$ implies the inequality $(x\rightarrow(x\cdot z))\rightarrow y\leq z\rightarrow y.$ Consequently we obtain
    \begin{eqnarray*}
        (x\cdot y)\wedge (x\cdot z) &=& x\cdot y\cdot ((x\cdot y)\rightarrow(x\cdot z))\\
        &=& x\cdot y\cdot (y\rightarrow (x\rightarrow(x\cdot z)))\\
        &=& x\cdot (x\rightarrow(x\cdot z))\cdot ((x\rightarrow(x\cdot z)) \rightarrow y)\\
        &\leq& x\cdot z\cdot (z\rightarrow y)\\
        &=& x\cdot (y\wedge z).
    \end{eqnarray*}
    The converse inequality is clear.
\end{proof}

The next lemma was proved in~\cite{GalTsin}.

\begin{lemma}[Galatos, Tsinakis]
    If $\mathbf A=(A;\cdot,\rightarrow,1)$ is hoop such that the induced ordered
    set $(A;\leq)$ has suprema $x\vee y$ for arbitrary $x,y\in A$ (especially if
    $\mathbf A$ is a finite hoop) then the induced lattice is distributive. 
\end{lemma}

We will continue to use the distributivity of the induced lattice in hoops (if $\vee$ exists) without mentioning.

We recall a few basic notions, properties and constructions related to hoops. 
An element $x$ is called \emph{idempotent} if $x\cdot x = x$.
We denote by $\mathbf{Id}\,\mathbf A$ the set of all idempotent elements of
$\mathbf A$.   

\begin{lemma}
If $\mathbf A=(A;\cdot,\rightarrow,1)$ is a hoop, $x\in \mathbf{Id}\, \mathbf A$
and $y\in \mathbf A$ then $x\cdot y=x\wedge y$. 
\end{lemma}
\begin{proof}
We have:
    \begin{eqnarray*}
        x\cdot y &=& x\wedge (x\cdot y) = (x\cdot x)\wedge (x\cdot y)\\
        &=& x\cdot (x\wedge y) = x\cdot x\cdot (x\rightarrow y)\\
        &=& x\cdot (x\rightarrow y)= x\wedge y.
    \end{eqnarray*}
\end{proof}

A \emph{filter} of a hoop $\mathbf{A}$ is a nonempty set $F\subseteq A$ satisfying
\begin{itemize}
    \item[(F1)] if $x\in F$ and $x\leq y$ then $y\in F,$
    \item[(F2)] if $x,y\in F$ then $x\cdot y\in F.$
\end{itemize}
We write $\mathbf{Fil}\,\mathbf A$ for the set of all filters of $\mathbf A$.
Any filter $F$ induces a congruence
$$
\theta_F=\{\langle x,y\rangle\in A^2\mid (x\rightarrow y)\cdot (y\rightarrow
x)\in F\}
$$
so we can, and will, write $\mathbf A/F$ for the factor algebra $\mathbf
A/\theta_F$.
Conversely if $\theta$ is a congruence of the hoop $\mathbf A$ then the class
$1/\theta$ is a filter and this correspondence between filters and
congruences is a lattice isomorphism.  

Suprema generally do not exist exist in hoops but if they exist (in particular,
if $\mathbf A$ is a finite hoop) then they are preserved by factorization. More
precisely, if $x\vee y\in\mathbf A$ exists then  $(x/F) \vee (y/F)\in \mathbf
A/F$ exists and $(x\vee y)/F=(x/F)\vee (y/F)$. 

It will be useful to have a shorthand notation for disjoint unions.
For an indexed system of sets $(A_i)_{i\in I}$, we write
$\sum_{i\in I}A_i$ for its disjoint union,
so that $(a,i)\in\sum_{i\in I} A_i$ if and only if $i\in I$ and $a\in A_i$.
For a disjoint union of two sets $A$ and $B$ we write $A+B$. 

For any $\mathbf{A}$ and $x\in A$ we put
$(x]=\{a\in A\mid a\leq x\}$. If $\mathbf{A}$ has a lattice reduct, then $(x]$
is a lattice ideal.  

Finally, we recall the definition of \emph{ordinal sum} of hoops.
Let $\mathbf A=(A;\cdot_\mathbf A,\rightarrow_\mathbf A,1)$ and $\mathbf
B=(B;\cdot_\mathbf B,\rightarrow_\mathbf B,1)$ be hoops.
The ordinal sum $\mathbf A\oplus\mathbf B$ is the hoop
$$\mathbf A\oplus\mathbf B=(A\setminus\{1\}+B,\rightarrow,\cdot,1)$$ in which
$$x\cdot y=\left\{\begin{array}{cll}
    x\cdot_\mathbf A y & \text{ iff }& x,y\in A\setminus\{1\}, \\
    x\cdot_\mathbf B y & \text{ iff }& x,y\in B,  \\
    x & \text{ iff }& x\in A\setminus\{1\},y\in B,  \\
    y & \text{ iff }& x\in B, y\in A\setminus\{1\}
    \end{array}\right.$$
and
$$x\rightarrow y=\left\{\begin{array}{cll}
    x\rightarrow_\mathbf A y & \text{ iff }& x,y\in A\setminus\{1\}, \\
    x\rightarrow_\mathbf B y & \text{ iff }& x,y\in B,  \\
    1 & \text{ iff }& x\in A\setminus\{1\},y\in B,  \\
    y & \text{ iff }& x\in B, y\in A\setminus\{1\}.
    \end{array}\right.$$
Roughly speaking, $\mathbf A\oplus\mathbf B$ arises by replacing the
element $1$ in $\mathbf A$, by (a copy of) $\mathbf{B}$.

We will conclude this section with one more remark on notation. If we take any hoop $\mathbf A$, then  if it is convenient we denote by $|\mathbf A|$ the underlying set of the hoop $\mathbf A.$
\section{$f$-product of hoops}
The following definition may not seem natural, but is the result of a process of
evolution starting from simple observations.
Let $\mathbf A=(A;\cdot ,\rightarrow,
1)$ be a finite hoop and let $F\in \mathbf{Fil}\,\mathbf A$.
Since every filter of a hoop $\mathbf A$ is also a subalgebra of $\mathbf A$,
we will distinguish notationally between a filter $F$ and a subalgebra
$\mathbf{F}$ with universe $F$.

A \emph{nucleus} on a hoop $\mathbf A$ is a closure operator
$\gamma\colon A\longrightarrow A$ satisfying $\gamma (a)\cdot \gamma (b)\leq
\gamma(a\cdot b)$. A \emph{nucleus image}  
of $\mathbf{A}$ (via $\gamma$) is the algebra
$\mathbf A_\gamma = (A_\gamma,\cdot_\gamma, \rightarrow, 1)$,
where $A_\gamma =\{\gamma(a)\mid a\in A\}$ and
$a\cdot_\gamma b= \gamma(a\cdot b)$. It is easy to show that
$\mathbf A_\gamma$ is also a hoop. Nuclei and nucleus images are quite important
in residuated-lattice theory, but we have no room for details here, and we refer
the reader to~\cite{GalTsin}.

For any $X\in \mathbf A/F$, we denote the top element of $X$ by $t_X$. A simple
application of divisibility shows that there are adjoint mappings  
\begin{center}
\begin{tikzcd}[row sep=huge]
 F \arrow[r,shift left=2,"t_X \cdot \-- "] \arrow[r,leftarrow, shift right=2,"\perp", " t_X\rightarrow \--"' ]&X 
 \end{tikzcd}
\end{center}
between the sets $F$ and $X$ such that
$$
t_X \cdot \-- \colon F\longrightarrow X
$$ is surjective and
$$
t_X\rightarrow \--\colon X\longrightarrow F
$$
is injective.
The composition of the mappings
$$
\gamma_X\colon F\longrightarrow F
$$
defined by $\gamma_X(a)=t_X\longrightarrow (t_X\cdot a)$ for any $a\in F$ is a
nucleus on $\mathbf F$. We write $\mathbf F_X$ (instead of $\mathbf
F_{\gamma_X}$) for the corresponding 
nucleus image. It is straightforward to show that $t_X\cdot \--\colon
F_X\longrightarrow X$ is a bijection. Then, there exists a
bijection $$\textstyle\Gamma\colon A\longrightarrow\sum_{X\in\mathbf A/F}F_X$$
defined by $\Gamma (a)=(t_X\rightarrow a,a/F)$. My first intention was
to transfer the original hoop structure from $\mathbf A$ to the set
$\sum_{X\in\mathbf A/F}F_X$ and try to describe the original hoop $\mathbf A$ as
a disjoint union of the nucleus images of the filter $F$ over the indexed system of
nuclei $(\gamma_X)_{X\in \mathbf A/F}$. However, the resulting structure was
syntactically complex and did not seem to bring any benefit. One more simple
observation proved crucial.  If we have a nucleus $\gamma_X$,
then the closed elements of the nucleus image $F_X$ are precisely of the form
$t_X\rightarrow a$ (for any $a\in F$). Moreover, if we denote by $l$ the
smallest element of the filter $F$, then $t_X\rightarrow a=t_X\rightarrow b$ if
and only if $a\wedge (t_X\vee l)=b\wedge (t_X\vee l)$.
The proof of this observation is a simple exercise and
has no further than motivational significance.

However, as a consequence of this simple observation, we get that
there exists
a bijection $t_X\rightarrow\--\colon (t_X\vee l]\longrightarrow F_X$ and hence
a bijection  
$$
\textstyle\Gamma'\colon A\longrightarrow\sum_{X\in\mathbf A/F}(t_X\vee l]
$$
which we will describe precisely below. The rest of the paper focusses on showing that
the transfer of structure from $\mathbf A$ to 
$\sum_{X\in\mathbf A/F}(t_X\vee l]$ is very natural and the key to it is
the mapping $X\mapsto l\vee t_X$ from $\mathbf A/F$ to $\mathbf F$. It will also
become clear that the following definition is nothing more than an algebraic
definition of this mapping.

\begin{defin}\label{def:prod-morph}
Let $\mathbf A=(A;\cdot,\rightarrow,1)$ and $\mathbf B=(B;\cdot, \rightarrow, 1)$
be hoops. Then the mapping $f\colon A\longrightarrow B$ we call a \emph{product
  morphism} from $\mathbf A$ to $\mathbf B$ if it satisfies 
\begin{itemize}
\item[(pM1)] $f(1)=1,$ 
\item[(pM2)] $f(x)\cdot f(y)= f(x\cdot y)=f(x)\wedge f(y)=f(x\wedge y)$.
\end{itemize}
for any $x,y\in A$.
\end{defin}

Preservation of infima gives the monotonicity of product morphisms. Using
the adjoint property and (pM2) we obtain the inequality $$f(x\rightarrow y)\leq
f(x)\rightarrow f(y)$$ for any product morphism $f.$ Moreover it is clear that
$f(x)\cdot f(x)=f(x)\wedge f(x)=f(x)$ and hence $f(\mathbf A)\subseteq
\mathbf{Id}\, \mathbf B$.  

\begin{defin}
Let $\mathbf A=(A;\cdot,\rightarrow,1)$ and $\mathbf B=(B;\cdot\rightarrow, 1)$
be hoops and let $f\colon B\longrightarrow A$ be a product morphism. The
algebra 
$$
\textstyle \mathbf A\ltimes_f \mathbf B =
(\sum_{x\in B}(f(x)],\cdot,\rightarrow,(1,1))
$$
where
\begin{itemize}
    \item[($\cdot$)] $(a,x)\cdot(b,y):=(a\cdot b,x\cdot y)$,
    \item[($\rightarrow$)] $(a,x)\rightarrow(b,y):=(f(x\rightarrow y)\wedge
      (a\rightarrow b), x\rightarrow y)$ 
\end{itemize}
for any $(a,x),(b,y)\in \sum_{x\in B}(f(x)]$,
will be called an \emph{$f$-product} of $\mathbf A$ and $\mathbf B$.
\end{defin}

\begin{theor}
If $\mathbf A=(A;\cdot,\rightarrow),1$ and $\mathbf B=(B;\cdot\rightarrow, 1)$
are hoops and $f\colon B\longrightarrow A$ is a product morphism, then the
$f$-product $\mathbf A\ltimes_f \mathbf B$ is a hoop.
\end{theor}

\begin{proof}
It is clear that $(a,x)\in  \sum_{x\in B}(f(x)]$ if and only if $a\leq f(x)$
(for any $(a,x)\in A\times B $) and hence $(a,x),(b,y)\in \sum_{x\in B}(f(x)]$
implies that $a\leq f(x)$ and $b\leq f(y)$ and consequently $a\cdot b\leq
f(x)\cdot f(y)=f(x\cdot y).$ Clearly also $f(x\rightarrow y)\wedge (a\rightarrow
b)\leq f(x\rightarrow y)$ and hence
$$
\textstyle (a,x)\cdot(b,y),
(a,x)\rightarrow(b,y)\in \sum_{x\in B}(f(x)].
$$
It is clear that $(\sum_{x\in B}(f(x)],\cdot,(1,1))$ is a commutative monoid and that
$$
(a,x)\rightarrow (a,x)=
(f(x\rightarrow x)\wedge (a\rightarrow a),x\rightarrow x)=(1,1)
$$
holds for any $(a,x)\in  \sum_{x\in B}(f(x)]$, showing that (H1) holds.

For (H2), note that 
if  $(a,x),(b,y),(c,z)\in  \sum_{x\in B}(f(x)]$ then $a\leq f(x)$ implies
$$
f(x\rightarrow (y\rightarrow z))\leq f(x)\rightarrow f(y\rightarrow z)\leq
a\rightarrow f(y\rightarrow z)
$$ 
and hence
    \begin{eqnarray*}
       && (a,x)\rightarrow ((b,y)\rightarrow (c,z)) \\
       &=& (a,x)\rightarrow (f(y\rightarrow z)\wedge (b\rightarrow c),y\rightarrow z)\\
       &=& (f(x\rightarrow (y\rightarrow z))\wedge (a\rightarrow (f(y\rightarrow z)\wedge (b\rightarrow c))),x\rightarrow (y\rightarrow z))\\
       &=& (f(x\rightarrow (y\rightarrow z))\wedge (a\rightarrow f(y\rightarrow z))\wedge (a\rightarrow (b\rightarrow c)),(x\cdot y)\rightarrow z)\\
       &=& (f((x\cdot y)\rightarrow z)\wedge ((a\cdot b)\rightarrow c),(x\cdot y)\rightarrow z)\\
       &=& (a\cdot b,x\cdot y)\rightarrow (c,z)\\
       &=&((a,x)\cdot (b,y))\rightarrow (c,z).  
    \end{eqnarray*}
To show that (H3) holds, we proceed similarly. Note that    
$b\leq f(y)\leq f(x\rightarrow y)$ and the idempotence of $f(x\rightarrow y)$ imply
    \begin{eqnarray*}
        (a,x)\cdot ((a,x)\rightarrow (b,y))&=&(a,x)\cdot (f(x\rightarrow y)\wedge (a\rightarrow b),x\rightarrow y)\\
        &=&(a\cdot f(x\rightarrow y)\wedge a\cdot (a\rightarrow b),x\cdot (x\rightarrow y))\\
        &=& (a\wedge b\wedge f(x\rightarrow y),x\wedge y)\\
        &=& (a\wedge b,x\wedge y).
    \end{eqnarray*}
By analogy, it is also true that $$(b,y)\cdot ((b,y)\rightarrow(a,x))=(b\wedge a,y\wedge x).$$
    From this it easily follows the divisibility 
    $$(a,x)\cdot ((a,x)\rightarrow (b,y))=(b,y)\cdot ((b,y)\rightarrow(a,x)).$$
\end{proof}

It is an easy observation that  $(a,x)\leq (b,y)$ holds if and only if
$$
(f(x\rightarrow y)\wedge (a\rightarrow b),x\rightarrow y)=
(a,x)\rightarrow (b,y)=(1,1).
$$
This is further equivalent to $a\leq b$ and $x\leq y$ for any
$(a,x),(b,y)\in |\mathbf A\ltimes_f\mathbf B|$.
Since $a\wedge b\leq f(x)\wedge f(y)=f(x\wedge y)$ and $a\vee b\leq
  f(x)\vee f(y)\leq f(x\vee y)$ hold for any $(a,x),(b,y)\in |\mathbf
  A\ltimes_f\mathbf B|$, we get that $(a\wedge b,x\wedge y),(a\vee b,x\vee y)\in
  |\mathbf A\ltimes_f\mathbf B|$.
Now, it is easy to derive the equations
$$
(a,x)\wedge (b,y)=(a\wedge b,x\wedge y)
$$
and
$$
(a,x)\vee (b,y)=(a\vee b,x\vee y)
$$
whenever $a\vee b$ and $x\vee y$ exist. It shows that the poset (lattice)
reduct of $\mathbf A\ltimes_f \mathbf B$ is a subposet (sublattice) of the poset (lattice) reduct of the direct product of $\mathbf A$ and $\mathbf B$.

\begin{exam}
If $\mathbf A$ and $\mathbf B$ are arbitrary hoops and $\varepsilon\colon
\mathbf B\longrightarrow\mathbf A$ is the constant mapping $\varepsilon (x)=1$
then clearly $\varepsilon$ is a product morphism and it is the greatest morphism
with respect to the natural order. It is easy to check that $$\mathbf
A\ltimes_\varepsilon\mathbf B= \mathbf A\times \mathbf B.$$ 
\end{exam}

\begin{exam}
Let $\mathbf A$ and $\mathbf B$ be hoops such that there exists the least
element $0\in \mathbf A$. Then the map $\sigma\colon
\mathbf B\longrightarrow\mathbf A$ defined by 
$$\sigma (x)=\left\{\begin{array}{cll}
1     &\text{ iff } & x=1  \\
0    & \text{ iff } & x\not =1  
\end{array}\right.
$$
is a product morphism. Clearly $\sigma$ is the least morphism
with respect to the natural order
and it is easy to check that
$$
\mathbf A\ltimes_\sigma\mathbf B= \mathbf B\oplus
\mathbf A.
$$ 
\end{exam}

\begin{defin}
If $\mathbf A,\mathbf B$ and $\mathbf C$ are hoops then we say that $(f,g)$ is a
\emph{left-associated pair of product morphims with respect to $(\mathbf
  A,\mathbf B,\mathbf C)$} if $f\colon \mathbf B\rightarrow \mathbf A$
and $$g\colon \mathbf C\longrightarrow \mathbf A\ltimes_f\mathbf B$$ are product
morphisms.  
Similarly, $(f,g)$ is a \emph{right-associated pair of product morphims with
  respect to $(\mathbf A,\mathbf B,\mathbf C)$} if $g\colon \mathbf C\rightarrow
\mathbf B$ and $$f\colon \mathbf B\ltimes_g\mathbf C\longrightarrow \mathbf A$$
are product morphisms. 
\end{defin}

It is obvious that any left-associated pair of product morphsims with respect to
$(\mathbf A,\mathbf B,\mathbf C)$ allows to construct the hoop $(\mathbf
A\ltimes_f\mathbf G)\ltimes_g\mathbf C$ and any right-associated pair of product
morphims with respect to $(\mathbf A,\mathbf B,\mathbf C)$ allows to construct
the hoop $\mathbf A\ltimes_f(\mathbf G\ltimes_g\mathbf C).$ The next
theorem describes the relationship between these concepts. 

\begin{theor}
(i) Let $(f,g)$ be a left-associated pair of product morphisms with respect
to  $(\mathbf A,\mathbf B,\mathbf C)$. Let $g_1$, $g_2$ be the natural
projection maps arising from $g$, so that $g(c)=(g_1(c),g_2(c))$
for any $c\in C$. Define the mappings: 
\begin{itemize}
\item $\overline g\colon \mathbf C\longrightarrow \mathbf B$ by
  $\overline{g}(c)=g_2(c)$ for all $c\in C$, 
\item $\overline{f}\colon \mathbf B\ltimes_{\overline{g}}\mathbf
  C\longrightarrow \mathbf A$ by $\overline{f}(b,c)=f(b)\wedge g_1(c)$ for any
  $(b,c)\in |\mathbf B\ltimes_{\overline g} \mathbf C|$. 
\end{itemize}  
Then $\alpha (f,g)=(\overline{f},\overline{g})$ is a right-associated pair
of product morphims with respect to $(\mathbf A,\mathbf B,\mathbf C)$. 

(ii) Let $(f,g)$ be a right-associated pair of product morphisms with respect to
$(\mathbf A,\mathbf B,\mathbf C)$. Define the mappings: 
\begin{itemize}
\item $\overline f\colon \mathbf B\longrightarrow \mathbf A$ by
  $\overline{f}(b)=f(b,1)$ for any $b\in B$, 
\item  $\overline{g}\colon \mathbf C\longrightarrow
  \mathbf A\ltimes_{\overline f}\mathbf B$ by
  $\overline g (c)=(f(g(c),c),g(c))$ for any $c\in C$.
\end{itemize}
Then $\beta (f,g)=(\overline{f},\overline{g})$ is a left-associated pair of
product morphisms with respect to  $(\mathbf A,\mathbf B,\mathbf C)$. 

(iii) The corespondences $\alpha$ and $\beta$ between left and right-associated
pairs of product morphisms with respect to  $(\mathbf A,\mathbf B,\mathbf C)$
are mutually inverse bijective mappings.  Moreover,
if $\alpha (f,g)=(\overline{f},\overline g)$, or equivalently $(f,g)=\beta
(\overline{f},\overline g)$, then
$$
(\mathbf A\ltimes_f\mathbf B)\ltimes_g\mathbf C\cong \mathbf
A\ltimes_{\overline f} (\mathbf B\ltimes_{\overline g}\mathbf C).
$$ 
\end{theor}

\begin{proof}
    (i) Since $f$ and $g$ are product morphisms and as the operations
    $\cdot$ and $\wedge$ are defined coordinatewise, we have
    \begin{eqnarray*}
     &&1 = f(1) =g_1(1)=g_2(1), \\
       && f(a\cdot b)= f(a\wedge b)=f(a)\cdot f(b)=f(a)\wedge f(b,)\\
       && g_1(x\cdot y)= g_1(x\wedge y)=g_1(x)\cdot g_1(y)=g_1(x)\wedge g_1(y),\\
       && g_2(x\cdot y)= g_2(x\wedge y)=g_2(x)\cdot g_2(y)=g_2(x)\wedge g_2f(y)  
    \end{eqnarray*}
    for all $a,b\in B$ and all $x,y\in C$. The fact that $\overline g:=g_2\colon
    \mathbf C\longrightarrow \mathbf B$ is a product morphism is a direct
    consequence of that. 
 
Next, since the element $f(a)$ is an idempotent, we have
$$
\overline{f}(b,c)=f(b)\wedge g_1(c)=f(b)\cdot g_1(c)
$$
for any $(b,c)\in |\mathbf B\ltimes_g \mathbf C|$. Hence
\begin{eqnarray*}
        \overline{f}((b_1,c_1)\cdot (b_2,c_2))&=&\overline{f}(b_1\cdot b_2,c_1\cdot c_2)=f(b_1\cdot b_2)\cdot g_1(c_1\cdot c_2)\\
        &=& f(b_1)\cdot f(b_2)\cdot g_1(c_1)\cdot g_1(c_2),\\
        &=& \overline f (b_1,c_1)\cdot \overline f(b_2,c_2),
\end{eqnarray*}
\begin{eqnarray*}
        \overline{f}((b_1,c_1)\wedge (b_2,c_2))&=&\overline{f}(b_1\wedge b_2,c_1\wedge c_2)=f(b_1\wedge b_2)\wedge g_1(c_1\wedge c_2)\\
        &=& f(b_1)\wedge f(b_2)\wedge g_1(c_1)\wedge g_1(c_2),\\
        &=& \overline f (b_1,c_1)\wedge \overline f(b_2,c_2),
\end{eqnarray*}
\begin{eqnarray*}
        \overline{f}((b_1,c_1)\wedge (b_2,c_2))&=&\overline{f}(b_1\wedge b_2,c_1\wedge c_2)=f(b_1\wedge b_2)\cdot g_1(c_1\wedge c_2)\\
        &=& f(b_1)\cdot f(b_2)\cdot g_1(c_1)\cdot g_1(c_2),\\
        &=& \overline f (b_1,c_1)\cdot \overline f(b_2,c_2),
\end{eqnarray*}
    for any $(b_1,c_1),(b_2,c_2)\in |\mathbf B\ltimes_g \mathbf C|$. Together, the mapping $$\overline f\colon |\mathbf B\ltimes_{\overline g}\mathbf C|\longrightarrow \mathbf A$$ is a product morphism.

    (ii) If $(f,g)$ is right-associated pair of product morphisms with respect to the triple of hoops  $(\mathbf A,\mathbf B,\mathbf C)$ then $\overline f (1)= f(1,1)=1$ and
\begin{eqnarray*}
        && f(b_1\cdot b_2,1)={f}(b_1,1)\cdot {f}(b_2,1) = {f}(b_1\wedge b_2,1)={f}(b_1,1)\wedge {f}(b_2,1) \\
                &=& \overline{f}(b_1\cdot b_2)=\overline{f}(b_1)\cdot \overline{f}(b_2) = \overline{f}(b_1\wedge b_2)=\overline{f}(b_1)\wedge \overline{f}(b_2) 
\end{eqnarray*}
    for any $b_1,b_2\in B$. Hence $\overline f\colon \mathbf B\longrightarrow\mathbf A$ is a product morphism.

    If $c\in \mathbf C$ then $f(g(c),c)\leq f(g(c),1)=\overline f(g(c))$. This implies that $\overline g(c)=(f(g(c),c), g(c))\in |\mathbf A\ltimes_{\overline f}\mathbf B|.$ It is clear that $\overline{g}(1)=(f(g(1),1),g(1))= (1,1)$ holds and
    \begin{eqnarray*}
    \overline {g} (c_1\cdot c_2) &=& (f(g(c_1\cdot c_2),c_1\cdot c_2),g(c_1\cdot c_2))\\
    &=& (f(g(c_1)\cdot g(c_2),c_1\cdot c_2),g(c_1)\cdot g(c_2))\\
    &=& (f(g(c_1),c_1)\cdot f(g(c_2),c_2),g(c_1)\cdot g(c_2))\\
    &=& f((g(c_1),c_1),g(c_1))\cdot f((g(c_2),c_2),g(c_2)) \\
    &=&\overline g(c_1)\cdot \overline g (c_2,)
    \end{eqnarray*}

      \begin{eqnarray*}
    \overline {g} (c_1\wedge c_2) &=& (f(g(c_1\wedge c_2),c_1\wedge c_2),g(c_1\wedge c_2))\\
    &=& (f(g(c_1)\wedge g(c_2),c_1\wedge c_2),g(c_1)\wedge g(c_2))\\
    &=& (f(g(c_1),c_1)\wedge f(g(c_2),c_2),g(c_1)\wedge g(c_2))\\
    &=& f((g(c_1),c_1),g(c_1))\wedge f((g(c_2),c_2),g(c_2)) \\
    &=&\overline g(c_1)\wedge \overline g (c_2,)
    \end{eqnarray*}

      \begin{eqnarray*}
    \overline {g} (c_1\cdot c_2) &=& (f(g(c_1\cdot c_2),c_1\cdot c_2),g(c_1\cdot c_2))\\
    &=& (f((g(c_1),c_1)\cdot g(c_2),c_2),g(c_1)\cdot g(c_2))\\
    &=& (f((g(c_1),c_1),g(c_1))\wedge (f((g(c_2),c_2),g(c_2))\\
    &=&\overline {g} (c_1)\wedge \overline {g} (c_2)
    \end{eqnarray*}

    hold for any $c_1,c_2\in C$. This proves that $\overline{g}\colon\mathbf
    C\longrightarrow \mathbf A\ltimes_{\overline f} \mathbf B$ is a product
    morphism. 

(iii) Let $(f,g)$ be a left-associated pair of product morphisms with respect to
$(\mathbf A,\mathbf B,\mathbf C)$. Put $\alpha
(f,g)=(\overline f,\overline g)$ and $\beta (\overline f,\overline g)=(f',g')$.
We decompose $g$ into projections, so that
$g(c)=(g_1(c),g_2(c))\in |\mathbf A\ltimes_f\mathbf B|$. Then
$g_1(c)\leq f(g_2(c)).$ Further, we have
$$
f'(b)=\overline{f}(b,1)=f(b)\wedge
g_1(1)=f(b)
$$
for any $b\in \mathbf B.$ If $c\in\mathbf C$ then 
    \begin{eqnarray*}
        g'(c)&=& (\overline f (\overline g(c),c),\overline g(c)) = (\overline f ( g_2(c),c), g_2(c))\\
        &=& (f(g_2(c))\wedge g_1(c),g_2(c))=(g_1(c),g_2(c))\\
        &=& g(c)
    \end{eqnarray*}
holds and moreover $\beta\alpha (f,g)=(f',g')=(f,g).$

Now let $(f,g)$ be a right-associated pair of product morphisms with respect to
$(\mathbf A,\mathbf B,\mathbf C)$. Put $\beta (f,g)=(\overline f,\overline g)$
and $\alpha (\overline f,\overline g)=(f',g')$.
Then 
$$
g'(c)=\overline g_2(c)= \pi_2 (\overline g(c))=\pi_2 (f(g(c),g(c)),g(c))=g(c)
$$
holds for any $c\in\mathbf C$.
Similarly if $(b,c)\in |\mathbf B\ltimes_g\mathbf C|$ (and so $b\leq g(c)$), then
we have
    \begin{eqnarray*}
        f'(b,c) &=& \overline f(b)\wedge \overline g_1(c)\\
        &=& f(b,1)\wedge \pi_1 (f(g(c),c),g(c))\\
        &=& f(b,1)\wedge f(g(c),c)=f(b\wedge g(c),1\wedge c)\\
        &=& f(b,c).
    \end{eqnarray*}
    Hence $\alpha\beta(f,g)=(f',g')=(f,g)$ holds.

It remains to show the ``moreover'' part of (iii).
Let $(f,g)$ be a right-associated pair of product morphisms with respect to
$(\mathbf A,\mathbf B,\mathbf C)$, let 
$\beta(f,g)=(\overline f,\overline g)$, and pick any
$a\in\mathbf A$, $b\in\mathbf B$ and $c\in\mathbf C$. Then
    \begin{eqnarray*}
        && ((a,b),c)\in |(\mathbf A\ltimes_{\overline f} \mathbf B)\ltimes_{\overline g}\mathbf C)| \\
         \iff && a\leq \overline f (b)=f(b,1) \,\mathbin{\&}\, (a,b) \leq \overline g(c)=(f(g(c),c),g(c))\\
         \iff && a\leq f(b,1) \,\mathbin{\&}\, a \leq f(g(c),c) \,\mathbin{\&}\, b\leq g(c)\\
         \iff && a\leq f(b,1)\wedge f(g(c),c) = f(b,c) \,\mathbin{\&}\, b\leq g(c)\\
         \iff && (a,(b,c))\in |\mathbf A\ltimes_f(\mathbf B\ltimes_g\mathbf C)|.
    \end{eqnarray*}
Hence we can define a bijective mapping 
$$
\Gamma\colon |\mathbf A\ltimes_f(\mathbf B\ltimes_g\mathbf C)|\longrightarrow
|(\mathbf A\ltimes_{\overline f} \mathbf B)\ltimes_{\overline g}\mathbf C)|
$$
putting
$\Gamma (a,(b,c))=((a,b),c)$ for any $(a,(b,c))\in |\mathbf A\ltimes_f(\mathbf
B\ltimes_g\mathbf C)|$. 

For any $(a_1,(b_1,c_1)), (a_2,(b_2,c_2))\in |\mathbf A\ltimes_f(\mathbf
B\ltimes_g\mathbf C)|$ we have
    \begin{eqnarray*}
        \Gamma (a_1,(b_1,c_1))\cdot \Gamma (a_2,(b_2,c_2)) &=& ((a_1,b_1),c_1)\cdot ((a_2,b_2),c_2)\\
        &=& ((a_1\cdot a_2,b_1\cdot b_2),c_1\cdot c_2)\\
        &=& \Gamma(a_1\cdot a_2,(b_1\cdot b_2,c_1\cdot c_2))\\
        &=&\Gamma ((a_1,(b_1,c_1))\cdot (a_2,(b_2,c_2))) 
    \end{eqnarray*}
and
\begin{eqnarray*}
   && \Gamma (a_1,(b_1,c_1))\rightarrow \Gamma (a_2,(b_2,c_2)) \\
   &=&((a_1,b_1),c_1)\rightarrow ((a_2,b_2),c_2)\\
    &=&(\overline g(c_1\rightarrow c_2) \wedge ((a_1,b_1)\rightarrow (a_2,b_2)),c_1\rightarrow c_2)\\
    &=&(\overline g(c_1\rightarrow c_2) \wedge (\overline f(b_1\rightarrow b_2) \wedge (a_1\rightarrow a_2),b_1\rightarrow b_2),c_1\rightarrow c_2)\\
    &=&( (f(g(c_1\rightarrow c_2),c_1\rightarrow c_2),g(c_1\rightarrow c_2)) \wedge\\ && (\overline f(b_1\rightarrow b_2) \wedge (a_1\rightarrow a_2),b_1\rightarrow b_2),c_1\rightarrow c_2)\\
    &=& ((f(g(c_1\rightarrow c_2),c_1\rightarrow c_2)\wedge \overline f(b_1\rightarrow b_2)\wedge (a_1\rightarrow a_2),g(c_1\rightarrow c_2)\wedge \\ && (b_1\rightarrow b_2)),c_1\rightarrow c_2)\\
    &=& ((f(g(c_1\rightarrow c_2),c_1\rightarrow c_2)\wedge \ f(b_1\rightarrow b_2,1)\wedge (a_1\rightarrow a_2),g(c_1\rightarrow c_2)\wedge \\ &&(b_1\rightarrow b_2)),c_1\rightarrow c_2)\\
    &=& ((f(g(c_1\rightarrow c_2)\wedge (b_1\rightarrow b_2),(c_1\rightarrow c_2)\wedge 1)\wedge \\ &&(a_1\rightarrow a_2),g(c_1\rightarrow c_2)\wedge (b_1\rightarrow b_2)),c_1\rightarrow c_2)\\
    &=& ((f(g(c_1\rightarrow c_2)\wedge (b_1\rightarrow b_2),c_1\rightarrow c_2)\wedge (a_1\rightarrow a_2),g(c_1\rightarrow c_2)\wedge \\ &&(b_1\rightarrow b_2)),c_1\rightarrow c_2)\\
    &=& ((f((b_1,c_1)\rightarrow(b_2,c_2))\wedge (a_1\rightarrow a_2),g(c_1\rightarrow c_2)\wedge (b_1\rightarrow b_2)),c_1\rightarrow c_2)\\
    &=& \Gamma (f((b_1,c_1)\rightarrow(b_2,c_2))\wedge (a_1\rightarrow a_2),(g(c_1\rightarrow c_2)\wedge (b_1\rightarrow b_2),c_1\rightarrow c_2))\\
    &=& \Gamma (f((b_1,c_1)\rightarrow(b_2,c_2))\wedge (a_1\rightarrow a_2),(b_1,c_1)\rightarrow(b_2,c_2))\\
    &=& \Gamma ((a_1,(b_1,c_1))\rightarrow (a_2,(b_2,c_2)))
\end{eqnarray*}
showing that $\Gamma$ is an isomorphism. 
\end{proof}

\section{Decomposition of finite hoops}

In this section we show that every finite hoop is (up to isomorphism)
a finite $f$-product of finite MV-chains. We first observe that
in finite hoops suprema always exist: an important fact which we later use
without mention.

\begin{lemma}\label{lem:cases}
Let $\mathbf A$ be a finite hoop and let $F\in\mathbf{Fil}\,\mathbf A$.
Let $t_X$ be the top element and $l_X$ the bottom element of $X$,
for any $X\in \mathbf A/F$. The bottom element of $F$ we denote $l$ (instead of
$l_F$) for simplicity. The following hold:

\begin{itemize}
    \item[(i)] $t_X\bullet t_Y\leq t_{X\bullet Y}$ for any operation $\bullet$ belonging to the set $\{\vee,\wedge,\cdot,\rightarrow\},$
    \item[(ii)] $X\leq Y$ if and only if $t_X\leq t_Y,$
    \item[(iii)] $t_X\wedge t_Y=t_{X\wedge Y},$ 
    \item[(iv)] $t_X\rightarrow t_Y=t_{X\rightarrow Y},$ 
    \item[(v)] $t_X\cdot t_{X\rightarrow Y}=t_{X\wedge Y},$ 
    \item[(vi)] $a\rightarrow t_X=t_X$ and $t_X\cdot a=t_X\wedge a$  for any $a\in F,$
    \item [(vii)] $t_X\cdot l=t_X\wedge l= l_X.$
\end{itemize}
for any $X,Y\in \mathbf A/F.$
\end{lemma}
\begin{proof}
    (i) It is clear that $t_X\bullet t_Y\in X\bullet Y.$ This immediately gives $t_X\bullet t_Y\leq t_{X\bullet Y}$ for any operation $\bullet$ belonging to the set $\{\vee,\wedge,\cdot,\rightarrow\}$.
    
    (ii) If $X\leq Y$ then $t_X\leq t_X\vee t_Y\leq t_{X\vee Y}=t_Y$. The converse implication is clear.
    
    (iii) Because $X\wedge Y\leq X,Y$ then due to (ii) we have $t_{X\wedge Y}\leq t_X,t_Y$ and $t_{X\wedge Y}\leq t_X\wedge t_Y$.
    
    (iv) Due to (i) and (ii) we have $t_X\cdot t_{X\rightarrow Y}\leq t_{X\cdot (X\rightarrow Y)}\leq t_Y.$ Hence $t_{X\rightarrow Y}\leq t_X\rightarrow t_Y$.

    (v) We have $t_{X\wedge Y}=t_X\wedge t_Y=t_X\cdot(t_X\rightarrow t_Y)=t_X\cdot t_{X\rightarrow Y}$.

    (vi) We have $a\cdot t_X,a\rightarrow t_X\in X,$ hence $a\cdot
    t_X,a\rightarrow t_X\leq t_X$ and consequently we obtain $a\rightarrow
    t_X\leq t_X\leq a\rightarrow t_X$. Then $a\cdot t_X=a\cdot (a\rightarrow
    t_X)= a\wedge t_X$. 

    (vii) We have $l\cdot t_x\in X$ and then $l_X\leq l\cdot t_X.$ Analogously $t_X\rightarrow l_X\in F$ and thus $l\leq t_X\rightarrow l_X$ and consequently $l\cdot t_X\leq l_X$ holds.
    \end{proof}

\begin{theor}\label{decomp}
If $\mathbf A$ is a finite hoop and $F\in\mathbf{Fil}\,\mathbf A$ then the
mapping $$\psi \colon \mathbf A/F\longrightarrow\mathbf F$$ defined by $\psi
(X)=t_X\vee l$ is a product morphism and moreover we have 
    $$\mathbf A\cong \mathbf F\ltimes_\psi(\mathbf A/ F).$$
\end{theor}

\begin{proof}
First we show that $\psi$ is a product morphism. It is clear that $\psi(F)=l\vee
t_F=1$. Lemma~\ref{lem:cases}(ii) gives monotonicity of
$\psi$, and using the full force of Lemma~\ref{lem:cases}, we obtain 
    \begin{eqnarray*}
        \psi(X)\cdot\psi (Y)&=&\psi (X)\cdot (t_Y\vee l)= (\psi (X)\cdot t_Y) \vee (\psi (X)\cdot l)\\
        &=& (\psi (X)\wedge t_Y)\vee l = (\psi (X)\wedge t_Y)\vee (\psi (X)\wedge l)\\
        &=& \psi (X)\wedge (t_Y\vee l) =\psi(X)\wedge\psi (Y)\\
        &=& (t_X\vee l)\wedge (t_Y\vee l) = (t_X\wedge t_Y)\vee l\\
        &=&t_{X\wedge Y}\vee l =\psi (X\wedge Y)\geq \psi (X\cdot Y)= t_{X\cdot Y}\vee l\\
        &\geq & t_X\cdot t_Y\vee l = (t_X\vee l)\cdot (t_Y\vee l)=\psi (X)\cdot \psi (Y).
    \end{eqnarray*}
Define the mapping
$$
\textstyle \Omega\colon A\longrightarrow \sum_{X\in \mathbf A/F}(\psi (X)]
$$
by putting
$\Omega (a)=(a\vee l,a/F)$.
It is clear that $a\vee l\leq t_{a/F}\vee l =\psi (a/F)$ and hence $\Omega$
is well defined. 

\begin{claim}
The mapping $\Omega$ is a bijection such that $\Omega^{-1}(a,X)=a\wedge t_X$.
\end{claim}
    \begin{proof}
        If $a\in A$ then
        \begin{eqnarray*}
            \Omega^{-1}(\Omega(a))&=&\Omega^{-1}(a\vee l,a/F) = (a\vee l)\wedge t_{a/F}\\
            &=&(a\wedge t_{a/F})\vee (l\wedge t_{a/F})= a \vee l_{a/F}\\
            &=& a.
        \end{eqnarray*}
        
If $(a,X)\in \sum_{X\in \mathbf A/F} (\psi (X)]$ then $l\leq a\leq \psi (X)$ and
        \begin{eqnarray*}
            \Omega(\Omega^{-1}(a,X)) &=& \Omega (a\wedge t_X) = ((a\wedge t_X)\vee l, (a\wedge t_X)/F)\\
            &=& ((a\vee l)\wedge (t_X\vee l),X) = (a\wedge \psi (X),X)\\
            &=&(a,X).
        \end{eqnarray*}
This finishes the proof of the claim.
\end{proof}

Next, we need to show that $\Omega$ is a homomorphism. 
Take $a,b\in A$ and calculate
    \begin{eqnarray*}
        \Omega(a)\cdot\Omega(b) &=&(a\vee l,a/F)\cdot (b\vee l, b/F)= ((a\vee l)\cdot(b\vee l),(a/F)\cdot (b/F))\\
        &=&((a\cdot b) \vee (a\cdot l)\vee (l\cdot b)\vee l, (a\cdot b)/F)\\
         &=&((a\cdot b) \vee l, (a\cdot b)/F)=\Omega(a\cdot b)\\
    \end{eqnarray*}
showing that $\Omega$ preserved products.    

Arrow preservation requires a little more work.
First, by Lemma~\ref{lem:cases}(vi) we get that
$a\leq \psi(X)=t_X\vee l$ implies $a\vee b \leq t_X\vee l\vee b=t_X\vee b$,
for any $(a,X),(b,Y)\in \sum_{X\in \mathbf A/F} (\psi (X)]$.
Hence
    \begin{eqnarray*}
        (a\wedge t_X)\rightarrow b &=& ((a\wedge t_X)\vee b)\rightarrow b\\
        &=&((a\vee b)\wedge (t_X\vee b))\rightarrow b\\
        &=& (a\vee b)\rightarrow b =a\rightarrow b.
    \end{eqnarray*}
Using this, in the third equality below, we get
    \begin{eqnarray*}
        \Omega^{-1}(a,X)\rightarrow \Omega^{-1}(b,Y) &=& (a\wedge t_X)\rightarrow (b\wedge t_Y) \\
        &=& ((a\wedge t_X)\rightarrow b) \wedge ((a\wedge t_X)\rightarrow t_Y)\\
        &=& (a\rightarrow b) \wedge ((a\cdot t_X)\rightarrow t_Y)\\
         &=& (a\rightarrow b) \wedge (a\rightarrow( t_X\rightarrow t_Y)\\
           &=& (a\rightarrow b) \wedge (a\rightarrow t_{X\rightarrow Y})\\
           &=& (a\rightarrow b) \wedge  t_{X\rightarrow Y}\\
           &=& \psi(X\rightarrow Y)\wedge (a\rightarrow b)\wedge t_{X\rightarrow Y}\\
         &=&\Omega^{-1}((\psi(X\rightarrow Y)\wedge (a\rightarrow b),X\rightarrow Y)\\
           &=&\Omega^{-1}((a,X)\rightarrow (b,Y)).
    \end{eqnarray*}
This equality implies
$\Omega^{-1}(\Omega(a)\rightarrow\Omega (b)) = \Omega^{-1}(\Omega(a))\rightarrow
\Omega^{-1}(\Omega(b))=a\rightarrow b$ and, of course, also 
$$
\Omega(a)\rightarrow \Omega (b)=\Omega(a\rightarrow b)
$$
showing that $\Omega$ preserves arrow as well. 
\end{proof}

\begin{defin}
A hoop $\mathbf A$ is irreducible if $\mathbf A \cong \mathbf B\ltimes_f \mathbf
C$, for an arbitrary product morphism $f\colon \mathbf C\longrightarrow \mathbf
B$, implies that $\mathbf B$ or $\mathbf C$ is the trivial hoop.
If a hoop is not irreducible, we call it reducible.
\end{defin}

\begin{theor}\label{thm:irr-simp}
A finite hoop is irreducible if and only if it is simple.
\end{theor}

\begin{proof}
Theorem~\ref{decomp} shows that if a finite hoop is not simple then it is
reducible. The converse implication follows from the fact that $$\pi_2\colon
\mathbf A\ltimes_f\mathbf B\longrightarrow \mathbf B,$$ defined by
$\pi_2(a,x)=x$ for any $(a,x)\in |\mathbf A\ltimes_f\mathbf B|,$ is a surjective
homomorphism and consequently the set $\{(a,1)\mid a\in A\}\subseteq |\mathbf
A\ltimes_f\mathbf B|$ is a nontrivial filter if both $\mathbf A$ and $\mathbf B$
are nontrivial hoops. 
\end{proof}

The next theorem comes from~\cite{BlFef}. 

\begin{theor}[Blok, Ferreirim]\label{thm:fin-simp}
Simple finite hoops are precisely finite MV-chains.
\end{theor}

\begin{theor}
(i) For any finite hoop $\mathbf A$ there exists a (finite) sequences of finite
MV-chains $(\mathbf M_{i_1},\mathbf M_{i_2},\dots,\mathbf M_{i_n})$ and
appropriate product morphisms $(f_1,\dots,f_{n-1})$ such that 
$$
\mathbf A\cong \mathbf M_{i_1}\ltimes_{f_1}(\mathbf
    M_{i_2}\ltimes_{f_2}(\mathbf M_{i_3}\dots\ltimes_{f_{n-2}}(\mathbf
    M_{i_{n-1}}\ltimes_{f_{n-1}}\mathbf M_{i_n})\dots)).
$$
    
(ii) For any finite hoop $\mathbf A$ there exists a (finite) sequences of
finite MV-chains $(\mathbf M_{i_1},\mathbf M_{i_2},\dots,\mathbf M_{i_n})$ and
appropriate product morphisms $(g_1,\dots,g_{n-1})$ such that 
$$
\mathbf A\cong (((\mathbf M_{i_1}\ltimes_{g_1}\mathbf
M_{i_2})\ltimes_{g_2}\mathbf M_{i_3})\dots\ltimes_{g_{n-2}}\mathbf
M_{i_{n-1}})\ltimes_{g_{n-1}}\mathbf M_{i_n}.
$$
\end{theor}

\begin{proof}
Combining Theorems~\ref{decomp}, \ref{thm:irr-simp} and~\ref{thm:fin-simp}. 
\end{proof}  

It is clear that a finite hoop $\mathbf A$, whose maximal chain of idempotent
elements has length $n+1$, can be decomposed into exactly $n$ MV-chains. The
next theorem gives 
a simple description of an $f$-product such that one of the factors is a finite
MV-chain. To state it conveniently, we need some technicalities.
Let $\mathbf A$ be a finite hoop, and $\mathbf M$ be a finite MV-chain.
Note that $\mathbf{Id}\, \mathbf M = \{0,1\}$.
Let $P_{\mathbf{A}}^{\mathbf{M}}$ be the set of all product morphisms from $\mathbf{A}$ to
$\mathbf{M}$, and let $P_{\mathbf{M}}^{\mathbf{A}}$ be the set of all product
morphisms from $\mathbf{M}$ to $\mathbf{A}$.

\begin{theor}
Let $\mathbf A$ be a finite hoop. There are natural bijections
$$
\nu\colon P_{\mathbf{A}}^{\mathbf{M}}\to \mathbf{Id}\, \mathbf A,
$$
and
$$
\mu\colon P_{\mathbf{M}}^{\mathbf{A}}\to \mathbf{Id}\, \mathbf A.
$$
\end{theor}

\begin{proof}
Let $\psi\in P_{\mathbf{A}}^{\mathbf{M}}$.
Then, $\psi^{-1}(1)$ is a filter; by finiteness it has the least element,
say $e_\psi$, which is idempotent. If
$\psi'\colon \mathbf A\longrightarrow \mathbf M$ is a product morphism
and $\psi'\neq\psi$, then $e_{\psi'}\neq e_\psi$, so the natural map
$\nu\colon P_{\mathbf{A}}^{\mathbf{M}}\to \mathbf{Id}\, \mathbf A$ given by
$\psi\mapsto e_\psi$ is injective.

Conversely, if $f$ is an idempotent element of $\mathbf A$
then the mapping $\psi_f\colon \mathbf A\longrightarrow\mathbf M$ 
    defined by
    $$\psi_f(x)=\left \{\begin{array}{lll}
        1 &\text{ if } & f\leq x \\
        0 & \text{ if } &f\not\leq x 
    \end{array}\right. $$
is  a product morphism, that is, $\psi_f\in P_{\mathbf{A}}^{\mathbf{M}}$.
Moreover, the smallest element of $\psi_f^{-1}(1)$ is precisely $f$,
that is $e_{\psi_f} = f$, showing that $\nu$ is surjective, hence bijective, as claimed. 
Figure~\ref{F3} shows the $\psi_f$-product for this choice of $\psi_f$. 

    \begin{figure}
        \centering
        \includegraphics[width=0.5\linewidth]{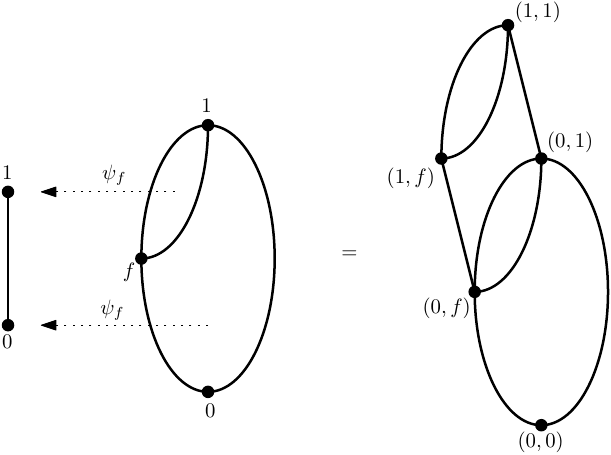}
        \caption{The visualisation of $\mathbf M\ltimes_{\psi_f}\mathbf A.$}
        \label{F3}
    \end{figure}

For the second case, we assign to each $\psi\in P_{\mathbf{M}}^{\mathbf{A}}$
the idempotent element $\psi(0)\in A$, that is, we put
$\mu(\psi) = \psi(0)$. So defined $\mu$ is injective, since for each $x\in M, \psi(x)=\psi(x)^n=\psi(x^n)=\psi(0)$. Conversely,
for $f\in\mathbf{Id}\;(\mathbf A)$, define 
$\psi_f\colon \mathbf M\longrightarrow\mathbf A$ by  
       $$\psi_f(x)=\left \{\begin{array}{lll}
        1 &\text{ if } & x=1, \\
        f & \text{ if } &x\not=1. 
       \end{array}\right. $$
Clearly, $\psi_f\in P_{\mathbf{M}}^{\mathbf{A}}$. Since $\mathbf{M}$ is simple, 
for any $x\in M$ with $x\not=1$ there exists $n\in\mathbb N$ such that
$x^n=0$. Hence $\psi_f(x)=\psi_f(x)^n=\psi_f(x^n)=\psi_f(0)$. Therefore, $\mu$
is surjective, hence bijective, as claimed. Figure~\ref{F4} the $\psi_f$-product for this choice of $\psi_f$.  
    \begin{figure}
        \centering
        \includegraphics[width=0.5\linewidth]{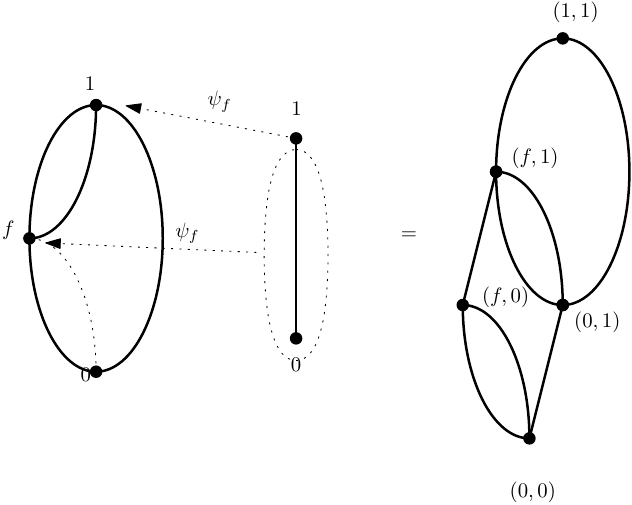}
        \caption{The visualisation of $\mathbf A\ltimes_{\psi_f}\mathbf M.$}
        \label{F4}
    \end{figure}
\end{proof}

\section{Appendix : The relation of $f$-products to exact sequences}

The last part of this article was inspired by a natural question: How to structurally interpret the composition of product morphisms? The answer, somewhat surprisingly, leads to exact sequences. Let's start with a simple observation. If $f\colon \mathbf A\longrightarrow \mathbf B$ is a homomorphism such that $f(\mathbf A)\in \mathbf{Fil}\; \mathbf B$ (let's call such homomorphisms \emph{filter homomorphisms}). Then Theorem \ref{decomp} says that there exists product morphisms $$\psi\colon \mathbf A/\mathrm{Ker}\,f\longrightarrow \mathrm{Ker}\,f$$ and $$\phi\colon \mathbf B/f(\mathbf A)\longrightarrow f(\mathbf A)$$ such that  
$$\mathbf A \cong (\mathrm{Ker}\, f)\ltimes_\psi (\mathbf A/\mathrm{Ker}\, f)$$
and
$$\mathbf B \cong f(\mathbf A)\ltimes_\phi(\mathbf B/f(\mathbf A)).$$
The first isomorphism theorem gives us the natural isomorphism $$\eta\colon \mathbf A/\mathrm{Ker}\,f \longrightarrow f(\mathbf A).$$ Together we get a product morphism $$\gamma\colon \mathbf B/f(\mathbf A)\longrightarrow \mathrm{Ker}\, f$$ that is the composition $\gamma := \psi\circ\eta^{-1}\circ\phi.$ Therefore, we can define the following hoop
$$\mathbf A\bullet_f\mathbf B:= (\mathrm{Ker}\, f)\ltimes_\gamma (\mathbf B/f(\mathbf A)).$$
Several properties can be seen immediately:
\begin{itemize}
    \item[i)] If we denote the (only) homomorphism $i\colon \mathbf 1\longrightarrow \mathbf A$ then we have $$\mathbf 1\bullet_i\mathbf A\cong \mathbf A.$$
    \item[ii)] If we denote the (only) homomorphism $j\colon \mathbf A\longrightarrow \mathbf 1$ then we have $$\mathbf A\bullet_j\mathbf 1\cong \mathbf A.$$
    \item[iii)] If $\mathbf F$ is a filter of the hoop $\mathbf A$ then $$\mathbf F\bullet_\subseteq \mathbf A\cong \mathbf A/F.$$ 
    \item[iv)] If $\mathbf F$ is a filter of the hoop $\mathbf A$ then $$\mathbf A\bullet_{-/F} (\mathbf A/F)\cong \mathbf F.$$ 
    \item[v)] The homomorphism $f\colon \mathbf A\longrightarrow\mathbf B$ is an isomorphism if and only if $$\mathbf A\bullet_f\mathbf B\cong\mathbf 1.$$ 
\end{itemize}

 An \emph{exact sequence of hoops} is a sequence of hoops and homomorphisms
$$\mathbf A_0\xlongrightarrow{f_1} \mathbf A_1\xlongrightarrow{f_2} \mathbf A_2\xlongrightarrow{f_3}\mathbf A_3\xlongrightarrow{f_4}\dots \xlongrightarrow{f_n}\mathbf A_n$$ 
such that we have $f_{i}(\mathbf A_{i})=\mathrm{Ker}\, f_{i+1}$ for arbitrary $i=1,\dots,n-1.$ If $$\mathbf A\xlongrightarrow{f}\mathbf B\xlongrightarrow{g} \mathbf C$$ is an exact sequence such that $g(\mathbf B)\in\mathbf{Fil}\,(\mathbf C)$ then Figure \ref{F5} visualise the situation where mapping $$f'\colon \mathrm{Ker}\,f\ltimes_\alpha f(\mathbf A)\longrightarrow f(\mathbf A)\ltimes_{\beta\circ\gamma}\mathbf C/g(\mathbf B)$$ is defined by $f'(a,b)=(b,1)$ and $$g'\colon \mathrm{Ker}\,f\ltimes_{\alpha\circ\beta}f(\mathbf B)\longrightarrow g(\mathbf B)\ltimes_\gamma \mathbf C/g(\mathbf B)$$ is defined by $g'(a,b)=(b,1),$ It is clear then that we can define the following $$\mathbf A\bullet_f\mathbf B\bullet_g\mathbf C:=\mathbf A\bullet_{f'}(\mathbf B\bullet_g\mathbf C)\cong(\mathbf A\bullet_f\mathbf B)\bullet_{g'}\mathbf C\cong\mathrm{Ker}\,f\ltimes_{\alpha\circ\beta\circ\gamma}\mathbf C/g(\mathbf B).$$ A deeper study of the concepts introduced in this appendix is beyond the scope of this article.
\begin{figure}
    \centering
    \includegraphics[width=0.9\linewidth]{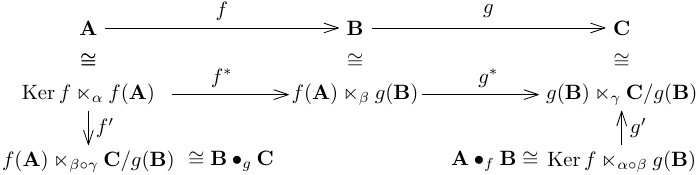}
    \caption{The composition of exact sequences}
    \label{F5}
\end{figure}

\section{Open problems}

In further research we want to focus on studying the uniqueness of existing decompositions and on studying homomorphisms of hoops determined by their decompositions. The question is also whether the introduced concepts and ideas can be extended to broader classes of residuated lattices.

\bibliographystyle{amsplain}
\bibliography{refer.bib}
\end{document}